\documentclass[11pt]{article}

\usepackage{amsfonts}
\usepackage{amssymb}
\usepackage{amsmath}
\usepackage{amsthm}
\def\negthickspace{\!\!\!}

\newcommand{\nicefrac}[2]
{\leavevmode \kern.1em\raise.5ex\hbox{\the\scriptfont0 #1}
             \kern-.1em/\kern-.15em\lower.25ex
             \hbox{\the\scriptfont0 #2}}

\newtheorem*{theorem}{Theorem}
\newtheorem*{proposition}{Proposition}
\newtheorem*{corollary}{Corollary}

\newtheorem*{definition}{Definition}
\newtheorem*{lemma}{Lemma}

\theoremstyle{definition}
\newtheorem*{remark}{Remark}
\newtheorem*{remarks}{Remarks}

\theoremstyle{definition}

\begin{document} 

\begin{center}
{\Large{\sc On critical normal sections}}\\[0.2cm]
{\Large{\sc for two-dimensional immersions in $\mathbb R^{n+2}$}}\\[1cm]
{\large Steffen Fr\"ohlich,\quad Frank M\"uller}\\[0.4cm]
{\small\bf Abstract}\\[0.4cm]
\begin{minipage}[c][2.5cm][l]{12cm}
{\small We study orthonormal normal sections of two-dimensional immersions in $\mathbb R^{n+2},$ $n\ge 2$, at which these sections are critical for a functional of total torsion. In particular, we establish upper bounds for the torsion coefficients in the case of non-flat normal bundles. With these notes we continue a foregoing paper on surfaces in $\mathbb R^4.$}
\end{minipage}
\end{center}
{\small MCS 2000: 53A07, 35J60, 49K20}


\section{Introduction}
\label{section1}
In the present paper we investigate orthonormal sections of normal bundles of two-dimensional immersions in Euclidian space $\mathbb R^{n+2}$, $n\ge 2,$ which are critical for the \emph{functional of total torsion}
\begin{equation*}
  {\mathcal T}_X({\mathcal N})
  =\sum_{\sigma,\vartheta=1}^n
   \int\hspace{-0.25cm}\int\limits_{\hspace{-0.3cm}B}
   h^{ij}T_{\sigma,i}^\vartheta T_{\sigma,j}^\vartheta\,W\,dudv;
\end{equation*}
see section \ref{section3} for the precise definition.\\[1ex]
Continuing the paper \cite{Froehlich_Mueller_01} on $2$-surfaces in $\mathbb R^4,$ we consider both, flat and non-flat normal bundles:
\begin{itemize}
\item[-]
For the flat case we prove that all torsion coefficients $T_{\sigma,i}^\vartheta$ of a critical orthonormal normal section have to vanish.
\vspace*{-1ex}
\item[-]
For the general case of normal bundles with non-vanishing curvature we derive estimates for the total torsion and the torsion coefficients of such sections.
\end{itemize}
The main difference to the case of immersions in $\mathbb R^4$ (that means $n=2$) is the nonlinear character of the pertinent differential equations appearing for $n\ge3$. Consequently, completely new arguments have to be applied. For instance, for $n=3$, all our results are based on the fact that a certain integral function of the torsion coefficients solves an (inhomogeneous) $H$-surface system; see section \ref{section4} for details.\\[1ex]
The outline of this paper is as follows:
\begin{itemize}
\item
In the next section we introduce the basic definitions of $2$-immersions $X$ in Euclidean spaces $\mathbb R^{n+2}$ for natural $n\ge 1.$
\vspace*{-1ex}
\item
In section 3 we introduce the functional of total torsion ${\mathcal T}_X({\mathcal N})$ and explain its significance for possible applications. Then, we compute the first variation of $\mathcal T_X(\mathcal N)$ w.r.t. $SO(n)$-perturbations as well as the associated Euler-Lagrange system. 
\vspace*{-1ex}
\item
Interpreting the Euler-Lagrange equations as integrability conditions, we derive a system of second order with quadratic growth in the gradient in section 4.
\vspace*{-1ex}
\item
In section 5 we prove that the torsion coefficients of critical orthonormal normal sections vanish identically, whenever the normal bundle is flat.
\vspace*{-1ex}
\item
Finally, in section 6 the reader finds a lower bound for the total torsion as well as an $L^\infty$-estimate for the torsion coefficients of critical normal sections, both for the general case of non-flat normal bundles in higher codimension.
\end{itemize}
\section{Basic settings and definition of the torsion}
\label{section2}
\setcounter{equation}{0}
\subsection{Two-dimensional immersions}
Let $n\ge 1$ be a natural number. We consider two-dimensional immersions
\begin{equation*}
  X=X(u,v)=\big(x^1(u,v),\ldots,x^{n+2}(u,v)\big)\in C^4(B,\mathbb R^{n+2}\,)
\end{equation*}
in the Euclidean space $\mathbb R^{n+2},$ parametrized on the closed unit disc
\begin{equation*}
  B=\Big\{(u,v)\in\mathbb R^2\,:\,u^2+v^2\le 1\Big\}\subset\mathbb R^2\,,
\end{equation*}
and such that the regularity condition
\begin{equation}\label{2.3}
  \mbox{rank}
  \begin{pmatrix} X_u(u,v)\\ X_v(u,v) \end{pmatrix}
  =2\quad\mbox{in}\ B
\end{equation}
is satisfied. Furthermore, we write $\mathring{B}=\big\{(u,v)\in\mathbb R^2\,:\,u^2+v^2<1\big\}$ for the open unit disc, its boundary is denoted by $\partial B=\big\{(u,v)\in\mathbb R^2\,:\,u^2+v^2=1\big\},$ and finally we set $B_\varrho(w_0):=\{w\in\mathbb R^2\,:\,|w-w_0|\le\varrho^2\}.$
\subsection{Conformal parametrization}
Let our immersions $X$ be conformally parametrized: That is, writing $X_{u^i}$ for the partial derivative of $X$ w.r.t. $u^i$ (where $u^1\equiv u$ and $u^2\equiv v$), there hold the \emph{conformality relations}
\begin{equation}\label{2.4}
  X_{u^i}\cdot X_{u^j}^t=:h_{ij}=W\delta_{ij}\quad\mbox{in}\ B
\end{equation}
for $i,j=1,2$. Here we used the area element
\begin{equation*}
  W:=\sqrt{h_{11}h_{22}-h_{12}^2}\,,
\end{equation*}
(note that $W>0$ in $B$ due to (\ref{2.3})) and the Kronecker symbol
\begin{equation*}
  \delta_{ij}
  :=\left\{
      \begin{array}{l}
        1\quad\mbox{if}\ i=j \\[0.1cm]
        0\quad\mbox{if}\ i\not=j
      \end{array}
    \right.
  \quad\mbox{for}\ i,j=1,2,
\end{equation*}
and, finally, $Z^t\in\mathbb R^d$ means the transposed vector of any $Z\in\mathbb R^d,$ $d\in\mathbb N$. As is well known, there is no restriction in assuming $X$ to be conformally parametrized,  see e.g. \cite{Sauvigny_01}. 
\subsection{Normal sections, torsion coefficients, and curvature \\ of the normal bundle}
Let ${\mathcal N}:=\big\{N_1,\ldots,N_n\big\}$ form an orthonormal section of the normal bundle of the immersion $X$ with the following properties:
\begin{equation*}
  N_\sigma\in C^3(B,\mathbb R^{n+2}\,),\quad
  X_{u^j}\cdot N_\sigma^t=0,\quad
  N_\sigma\cdot N_\vartheta^t=\delta_{\sigma\vartheta}
  \quad\mbox{in}\ B
\end{equation*}
for all $i=1,2$ and all $\sigma,\vartheta=1,\ldots,n.$
\goodbreak\noindent
To such a section ${\mathcal N}$ we associate the so-called torsion coefficients in the following sense:
\begin{definition}
The torsion coefficients of an orthonormal normal section ${\mathcal N}$ (shortly: ONS ${\mathcal N}$) are defined as
\begin{equation*}
  T_{\sigma,i}^\vartheta:=N_{\sigma,u^i}\cdot N_\vartheta^t\,,\quad i=1,2,\ \sigma,\vartheta=1,\ldots,n.
\end{equation*}
\end{definition}
\begin{remarks}\quad
\begin{itemize}
\item[1.]
Note the skew-symmetry of the torsion coefficients: 
\begin{equation*}
  T_{\sigma,i}^\vartheta=-T_{\vartheta,i}^\sigma
  \quad\mbox{for all}\ i=1,2,\ \sigma,\vartheta=1,\ldots,n.
\end{equation*}
\item[2.]
The $T_{\sigma,i}^\vartheta$ are exactly the coefficients of the normal connection (see e.g. \cite{doCarmo_01}), while our name {\it torsion coefficients} follows the one-dimensional theory of space curves. Namely, if ${\mathfrak n}$ and ${\mathfrak b}$ denote the normal resp. the binormal of an arc-length parametrized curve $c=c(s),$ then its torsion is defined as the inner product ${\mathfrak n}(s)'\cdot{\mathfrak b}(s)^t.$
\end{itemize}
\end{remarks}
\noindent
The coefficients $S_{\sigma,ij}^\vartheta\in C^1(B,\mathbb R)$ of the curvature tensor ${\mathfrak S}$ of the normal bundle are given by 
\begin{equation}\label{2.10}
  S_{\sigma,ij}^\vartheta
  :=T_{\sigma,i,u^j}^\vartheta-T_{\sigma,j,u^i}^\vartheta
   +T_{\sigma,i}^\omega T_{\omega,j}^\vartheta-T_{\sigma,j}^\omega T_{\omega,i}^\vartheta\,,\quad
  i,j=1,2,\ \sigma,\vartheta=1,\ldots,n
\end{equation}
(see again \cite{doCarmo_01}; summation convention for $\omega=1,\ldots,n$). Note that the $S_{\sigma,ij}^\vartheta$ are skew-symmetric in $i,j$ and $\sigma,\vartheta$. Consequently, they are completely described by the $N:=\frac{1}{2}n(n-1)$ quantities
\begin{equation}\label{2.11}
  S_{\sigma,12}^\vartheta\quad
  \mbox{for}\ (\sigma,\vartheta)\in U_n:=\Big\{(\omega,\delta)\in\{1,\ldots,n\}^2\,:\ \omega<\delta\Big\}.
\end{equation}
For example, in $\mathbb R^4$ there is -- up to the sign -- only one relevant quantity $S_{1,12}^2.$
\section{The total torsion and its properties}
\label{section3}
\setcounter{equation}{0}
\subsection{Definition of the total torsion}
In the paper at hand we study ONS ${\mathcal N}$ which are critical for the following {\it functional of total torsion} (summation convention for $i,j=1,2$)
\begin{equation}\label{3.1}
  {\mathcal T}_X({\mathcal N})
  :=\sum_{\sigma,\vartheta=1}^n\,
   \int\hspace*{-0.25cm}\int\limits_{\hspace{-0.3cm}B}
   h^{ij}\,T_{\sigma,i}^\vartheta T_{\sigma,j}^\vartheta\,W\,dudv,
\end{equation}
where the $h^{ij}$ are the elements of the inverse matrix to $(h_{ij})_{i,j=1,2}$ from (\ref{2.4}): 
\begin{equation*}
  h_{ij}h^{jk}=\delta_i^k\quad\mbox{in}\ B\quad \mbox{for}\ i,k=1,2.
\end{equation*}
\begin{remark}
The total torsion ${\mathcal T}_X$ does not depend on the parametrization of $X$, but it depends on the chosen ONS ${\mathcal N}.$
\end{remark}
\noindent
Taking the conformal parametrization (\ref{2.4}) of $X$ into account and using the definition of $U_n$ from (\ref{2.11}), ${\mathcal T}_X$ takes the form
\begin{equation}\label{3.3}
  {\mathcal T}_X({\mathcal N})
  =2\sum_{(\sigma,\vartheta)\in U_n}\,\int\hspace*{-0.25cm}\int\limits_{\hspace{-0.3cm}B}
    \Big\{(T_{\sigma,1}^\vartheta)^2+(T_{\sigma,2}^\vartheta)^2\Big\}\,dudv.
\end{equation}
We want to establish bounds for this functional as well as for the torsions $T_{\sigma,i}^\vartheta$ of critical ONS ${\mathcal N},$ the latter in terms of the value of ${\mathcal T}_X$ itself and an $L^\infty$-bound for $S_{\sigma,12}^\vartheta.$
\subsection{Fields of application}
\begin{itemize}
\item[1.]
First, the total torsion appears in many concrete situations, for example, in the second variation formula of the area functional
\begin{equation*}
  {\mathcal A}[X]
  :=\int\hspace{-0.25cm}\int\limits_{\hspace{-0.3cm}B}\sqrt{h_{11}h_{22}-h_{12}^2}\,dudv.
\end{equation*}
Namely, choose an ONS ${\mathcal N}=\{N_1,\ldots,N_n\},$ and consider a normal variation $\widetilde X=X+\chi N_\omega$ of a conformally parametrized minimal surface $X,$ where $N_\omega\in{\mathcal N}$ and $\chi\in C_0^\infty(B,\mathbb R).$ For the second variation of ${\mathcal A}[X]$ w.r.t.~$N_\omega\in{\mathcal N}$ one then computes
\begin{equation*}
  \delta_{N_\omega}^2{\mathcal A}[X;\chi]
  =\int\hspace{-0.25cm}\int\limits_{\hspace{-0.3cm}B}(|\nabla\chi|^2+2K_{N_\omega}W\chi^2)\,dudv
   +\sum_{\sigma=1}^n
    \int\hspace{-0.25cm}\int\limits_{\hspace{-0.3cm}B}
    \Big\{
      (T_{\omega,1}^\sigma)^2+(T_{\omega,2}^\sigma)^2
    \Big\}\,\chi^2\,dudv
\end{equation*}
with the ``Gaussian curvature'' $K_{N_\omega}$ w.r.t. $N_\omega$ (see e.g. \cite{Froehlich_01}). Therefore, it is desirable to control the torsion coefficients of suitable chosen ONS ${\mathcal N}.$
\item[2.]
Next, taking Ricci's integrability conditions
  $$S_{\sigma,12}^\vartheta=(L_{\sigma,1m}L_{\vartheta,2n}-L_{\sigma,2m}L_{\vartheta,1n})h^{mn},\quad
    L_{\sigma,ij}:=-X_{u^i}\cdot N_{\sigma,u^j}^t=X_{u^iu^j}\cdot N_{\sigma}^t\,,$$
into account (see e.g. \cite{doCarmo_01}), we can bound the curvature of the normal bundle in terms of the Gaussian curvature $K$ and the length of the mean curvature vector ${\mathcal H}$ of $X:$
  $$|S_{\sigma,12}^\vartheta|\le 2\,\big\{|{\mathcal H}|^2-K\big\}\,W.$$
Therefore, to control a geometric quantity in terms of $|S_{\sigma,12}^\vartheta|$ means to control it by means of $|{\mathcal H}|^2-K.$
\item[3.]
Finally, and more generally, the differential geometry of immersions with non-trivial normal bundles is certainly far away from beeing completely developed. This is manifested in the fact that many problems are satisfactorally solved only in the case of vanishing curvature tensor ${\mathfrak S}$ (see e.g. \cite{Ferapontov_01}, \cite{Smoczyk_Wang_Xin_01}), or if one restricts to special geometric situations (see e.g. \cite{Bergner_Froehlich_01} for curvature estimates for {\it graphs}).\\[0.6ex]
With this paper we aim at giving partial answers to questions like these:
\begin{itemize}
\item
What characteristic geometrical quantities of immersions can be controlled in terms of the curvature tensor ${\mathfrak S}?$
\vspace*{-0.2ex}
\item
What geometric properties share immersions with the same, possibly constant curvature of the normal bundle?
\end{itemize}
\end{itemize}
\subsection{The Euler-Lagrange equations}
We will derive the Euler-Lagrange equations for critical ONS ${\mathcal N}$. To this end, we consider a one-parameter family of rotations
\begin{equation*}
  {\mathbf R}(w,\varepsilon)
  =\big(r_{\sigma\vartheta}(w,\varepsilon)\big)_{\sigma,\vartheta=1,\ldots,n}
  \in C^\infty(B\times(-\varepsilon_0,+\varepsilon_0),SO(n)),
  \quad w=(u,v),
\end{equation*}
with sufficiently small $\varepsilon_0>0,$ such that
\begin{equation}\label{3.7}
  {\mathbf R}(w,0)={\mathbb E}^n\,,\quad
  \frac{\partial}{\partial\varepsilon}\,{\mathbf R}(w,0)={\mathbf A}(w)\in C^\infty(B,so(n)).
\end{equation}
Here, ${\mathbb E}^n$ denotes the $n$-dimensional unit matrix.\\[1ex]
For arbitrary skew-symmetric ${\mathbf A}(w)=(a_{\sigma\vartheta}(w))_{\sigma,\vartheta=1,\ldots,n}\in C^\infty(B,so(n))$, a family ${\mathbf R}(w,\varepsilon)$ with the property (\ref{3.7}) can be constructed by means of the geodesic flow in the manifold $SO(n)$ (see e.g. \cite{doCarmo_01}, chapter 3, section 2). By expansion around $\varepsilon=0$ we obtain
\begin{equation*}
  \mathbf R(w,\varepsilon)=\mathbb E^n+\varepsilon\mathbf A(w)+o(\varepsilon).
\end{equation*}
Now, we apply ${\mathbf R}$ to a given ONS ${\mathcal N}$. The new unit normal vectors $\widetilde N_1,\ldots,\widetilde N_n$ are given by
\begin{equation*}
  \widetilde N_\sigma
  =\sum_{\vartheta=1}^n r_{\sigma\vartheta}N_\vartheta
  =\sum_{\vartheta=1}^n
   \big\{
     \delta_{\sigma\vartheta}+\varepsilon a_{\sigma\vartheta}+o(\varepsilon)
   \big\}\,N_\vartheta
  =N_\sigma+\varepsilon\sum_{\vartheta=1}^na_{\sigma\vartheta}N_\vartheta+o(\varepsilon),
\end{equation*}
and we compute 
\begin{equation*}
  \widetilde N_{\sigma,u^\ell}
  =N_{\sigma,u^\ell}
   +\varepsilon\sum_{\vartheta=1}^n
    \big(
      a_{\sigma\vartheta,u^\ell}N_\vartheta+a_{\sigma\vartheta}N_{\vartheta,u^\ell}
    \big)
   +o(\varepsilon)
\end{equation*}
for their derivatives. Consequently, the new torsion coefficients can be expanded to
\begin{equation*}
  \widetilde T_{\sigma,\ell}^\omega=\widetilde N_{\sigma,u^\ell}\cdot\widetilde N_\omega^t
  =T_{\sigma,\ell}^\omega
   +\varepsilon a_{\sigma\omega,u^\ell}
   +\varepsilon\sum_{\vartheta=1}^n
    \big\{
      a_{\sigma\vartheta}T_{\vartheta,\ell}^\omega
      +a_{\omega\vartheta}T_{\sigma,\ell}^\vartheta
    \big\}
   +o(\varepsilon),
\end{equation*}
and for their squares we infer
\begin{equation*}
  (\widetilde T_{\sigma,\ell}^\omega)^2
  =(T_{\sigma,\ell}^\omega)^2
   +2\varepsilon
    \bigg\{
      a_{\sigma\omega,u^\ell}T_{\sigma,\ell}^\omega
      +\sum_{\vartheta=1}^n
       \big(
         a_{\sigma\vartheta}T_{\vartheta,\ell}^\omega T_{\sigma,\ell}^\omega
         +a_{\omega\vartheta}T_{\sigma,\ell}^\vartheta T_{\sigma,\ell}^\omega
       \big)
    \bigg\}
   +o(\varepsilon).
\end{equation*}
Before we insert this identity into the functional of total torsion, we observe
\begin{equation*}
\begin{array}{lll}
  \displaystyle
  \sum_{\sigma,\omega,\vartheta=1}^n
  \big\{
    a_{\sigma\vartheta}T_{\vartheta,\ell}^\omega T_{\sigma,\ell}^\omega
    +a_{\omega\vartheta}T_{\sigma,\ell}^\vartheta T_{\sigma,\ell}^\omega
  \big\}
  & = & \negthickspace\displaystyle
        \sum_{\sigma,\omega,\vartheta=1}^n
        \big\{
          a_{\sigma\vartheta}T_{\vartheta,\ell}^\omega T_{\sigma,\ell}^\omega
          +a_{\sigma\vartheta}T_{\omega,\ell}^\vartheta T_{\omega,\ell}^\sigma
        \big\} \\[4ex]
  & = & \negthickspace\displaystyle
        2\sum_{\sigma,\omega,\vartheta=1}^n
        a_{\sigma\vartheta}T_{\vartheta,\ell}^\omega T_{\sigma,\ell}^\omega
        \,=\,0,
\end{array}
\end{equation*}
taking the skew-symmetry of $\mathbf A$ into account.
\goodbreak\noindent
Now, the difference between the torsion functionals computes to ($a_{\sigma\omega,u^\ell}T_{\sigma,\ell}^\omega=a_{\omega\sigma,u^\ell}T_{\omega,\ell}^\sigma$)
\begin{equation*}
\begin{array}{lll}
  \displaystyle
  {\mathcal T}_X(\widetilde{\mathcal N})-{\mathcal T}_X({\mathcal N})\negthickspace
  & = & \negthickspace\displaystyle
        2\varepsilon
        \sum_{\sigma,\omega=1}^n\sum_{\ell=1}^2\,
        \int\hspace{-0.25cm}\int\limits_{\hspace{-0.3cm}B}
        a_{\sigma\omega,u^\ell}T_{\sigma,\ell}^\omega\,dudv
        +o(\varepsilon) \\[5ex]
  & = & \negthickspace\displaystyle
         4\varepsilon
         \sum_{1\le\sigma<\omega\le n}
         \int\hspace{-0.25cm}\int\limits_{\hspace{-0.3cm}B}
         \Big\{
           a_{\sigma\omega,u}T_{\sigma,1}^\omega
           +a_{\sigma\omega,v}T_{\sigma,2}^\omega
         \Big\}
        +o(\varepsilon)\\[5ex]
  & = & \negthickspace\displaystyle
        \,4\varepsilon
        \sum_{1\le\sigma<\omega\le n}\,
        \int\limits_{\partial B}
        a_{\sigma\omega}(T_{\sigma,1}^\omega,T_{\sigma,2}^\omega)\cdot\nu^t\,ds \\[5ex]
  &   & \negthickspace\displaystyle
        -\,4\varepsilon
        \sum_{1\le\sigma<\omega\le n}^n\,
        \int\hspace{-0.25cm}\int\limits_{\hspace{-0.3cm}B}
        a_{\sigma\omega}\,\mbox{div}\,(T_{\sigma,1}^\omega,T_{\sigma,2}^\omega)\,dudv
       +o(\varepsilon),
\end{array}
\end{equation*}
where $\nu$ denotes the outer unit normal of $\partial B.$ Thus, for critical ${\mathcal N}$ we infer
\begin{equation*}
  \sum_{1\le\sigma<\vartheta\le n}\,\,
  \int\limits_{\partial B}
  a_{\sigma\omega}(T_{\sigma,1}^\omega,T_{\sigma,2}^\omega)\cdot\nu^t\,ds
  -\sum_{1\le\sigma<\omega\le n}\,
      \int\hspace{-0.25cm}\int\limits_{\hspace{-0.3cm}B}
      a_{\sigma\omega}\,\mbox{div}\,(T_{\sigma,1}^\omega,T_{\sigma,2}^\omega)\,dudv
  =0
\end{equation*}
with arbitrary ${\mathbf A}\in C^\infty(B,so(n)).$ This implies the
\begin{proposition}
If the ONS ${\mathcal N}$ is critical for ${\mathcal T}_X,$ then its torsion coefficients satisfy
\begin{equation}\label{3.16}
  \mbox{\rm div}\,(T_{\sigma,1}^\vartheta,T_{\sigma,2}^\vartheta)=0\quad\mbox{in}\ B,\quad
  (T_{\sigma,1}^\vartheta,T_{\sigma,2}^\vartheta)\cdot\nu^t=0\quad\mbox{on}\ \partial B
\end{equation}
for all $(\sigma,\vartheta)\in U_n$.  
\end{proposition}
\section{A second order system with quadratic growth}
\label{section4}
\setcounter{equation}{0}
\subsection{The functions $g^{(\sigma\vartheta)}$}
Interpreting the differential equations in (\ref{3.16}) as integrability conditions, we find functions $g^{(\sigma\vartheta)}\in C^2(B,\mathbb R)$ such that 
\begin{equation}\label{4.1}
  \nabla g^{(\sigma\vartheta)}=\big(-T_{\sigma,2}^\vartheta,T_{\sigma,1}^\vartheta\big)
  \quad\mbox{in}\ B\quad\mbox{for all}\ \sigma,\vartheta=1,\ldots,n.
\end{equation}
Due to the boundary conditions in (\ref{3.16}), which imply $\nabla g^{(\sigma\vartheta)}\cdot\tau^t=0$ on $\partial B$ with the unit tangent vector $\tau=(-v,u)$ at $\partial B,$ we may choose $g^{(\sigma\vartheta)}$ such that
\begin{equation}\label{4.2}
  g^{(\sigma\vartheta)}=0\quad\mbox{on}\ \partial B\quad\mbox{for all}\ \sigma,\vartheta=1,\ldots,n.
\end{equation}
Note that the matrix $(g^{(\sigma\vartheta)})_{\sigma,\vartheta=1,\ldots,n}$ is skew-symmetric.
\goodbreak\noindent
\subsection{An elliptic system of second order}
Let us define the quantities
\begin{equation*}
  \delta g^{(\sigma\vartheta)}:=\sum_{\omega=1}^n\mbox{det}\,\Big(\nabla g^{(\sigma\omega)},\nabla g^{(\omega\vartheta)}\Big),
  \quad\sigma,\vartheta=1,\ldots n.
\end{equation*}
The matrix $(\delta g^{(\sigma\vartheta)})_{\sigma,\vartheta=1,\ldots,n}$ is skew-symmetric. The functions $g^{(\sigma\vartheta)}$ solve a coupled quasilinear elliptic differential system with quadratic growth in the gradient:
\begin{proposition}
If ${\mathcal N}$ is critical for ${\cal T}_X$, then the functions $g^{(\sigma\vartheta)}$, $\sigma,\vartheta=1,\ldots,n$, are solutions of the boundary value problems
\begin{equation}\label{4.4}
  \Delta g^{(\sigma\vartheta)}=-\,\delta g^{(\sigma\vartheta)}+S_{\sigma,12}^\vartheta\quad\mbox{in}\ B,\quad
  g^{(\sigma\vartheta)}=0\quad\mbox{on}\ \partial B\,,
\end{equation}
where $\delta g^{(\sigma\vartheta)}$ grows quadratically in the gradient of $g^{(\sigma\vartheta)}$.
\end{proposition}
\begin{proof}
Choose any $(\sigma,\vartheta)\in\{1,\ldots,n\}^2.$ The formulas (\ref{2.10}) and (\ref{4.1}) imply
\begin{equation}\label{4.5}
\begin{array}{rcl}
  \Delta g^{(\sigma\vartheta)}\negthickspace
  & = & \negthickspace\displaystyle
        T_{\sigma,1,v}^\vartheta -T_{\sigma,2,u}^\vartheta
        \,=\,-\sum_{\omega=1}^nT_{\sigma,1}^\omega T_{\omega,2}^\vartheta
             +\sum_{\omega=1}^nT_{\sigma,2}^\omega T_{\omega,1}^\vartheta
             +S_{\sigma,12}^\vartheta \\[4ex]
  & = & \negthickspace\displaystyle
        \sum_{\omega=1}^n
        \Big\{
          g_v^{(\sigma\omega)}g_u^{(\omega\vartheta)}-g_u^{(\sigma\omega)}g_v^{(\omega\vartheta)}
        \Big\}
        +S_{\sigma,12}^\vartheta\,,
\end{array}
\end{equation}
and the statement follows.
\end{proof}
\subsection{Examples}
Let us write $\mathbf G:=(g^{(\sigma\vartheta)})_{\sigma,\vartheta=1,\ldots,n}$, $\mathbf S:=(S_{\sigma,12}^\vartheta)_{\sigma,\vartheta=1,\ldots,n}$, and $\delta\mathbf G:=(\delta g^{(\sigma\vartheta)})_{\sigma,\vartheta=1,\ldots,n}$. We discuss the special cases $n=1,2,3.$
\begin{itemize}
\item[1.]
For $n=1$ ($X$ is immersed in $\mathbb R^3$) there are no torsions.
\vspace*{-1ex}
\item[2.]
The case $n=2$ ($X$ is immersed in $\mathbb R^4$) was already considered in \cite{Froehlich_Mueller_01}. There hold
\begin{equation}\label{4.6}
  {\mathbf S}
  =\begin{pmatrix}
     0         & S_{1,12}^2 \\[0.1cm]
     S_{2,12}^1 & 0
   \end{pmatrix},\quad
  {\mathbf G}
  =\begin{pmatrix}
     0       & g^{(12)} \\[0.1cm]
     g^{(21)} & 0
   \end{pmatrix},\quad
  \delta\mathbf G
  =\begin{pmatrix}
     0       & 0 \\[0.1cm]
     0 & 0
   \end{pmatrix},
\end{equation}
such that the system (\ref{4.4}) reduces to the single equation
\begin{equation*}
  \Delta g^{(12)}=S_{1,12}^2\quad\mbox{in}\ B,\quad g^{(12)}=0\quad\mbox{on}\ \partial B.
\end{equation*}
Then, potential theoretical estimates for elliptic equations ensure
\begin{equation*}
  \|g^{(12)}\|_{C^{1+\alpha}(B)}\le C(\alpha,\|S_{1,12}^2\|_\infty)
  \quad\mbox{for all}\ \alpha\in(0,1)
\end{equation*}
with a real $C\in(0,+\infty)$ depending on $\alpha$ and the $L^\infty$-norm of $S_{1,12}^2$ (see e.g. \cite{Sauvigny_02}).\\[1ex]
Instead of this, in \cite{Froehlich_Mueller_01} we study a Riemann-Hilbert problem for $T_{1,1}^2+iT_{1,2}^2$ using methods from the complex analysis of generalized analytic functions.
\vspace*{-1ex}
\item[3.]
Let us now consider the case $n=3$ ($X$ is immersed in $\mathbb R^5$): We have
\begin{equation*}
\begin{array}{l}
  \displaystyle
  {\mathbf S}
  =\begin{pmatrix}
     0         & S_{1,12}^2 & S_{1,12}^3 \\[0.1cm]
     S_{2,12}^1 & 0         & S_{2,12}^3 \\[0.1cm]
     S_{3,12}^1 & S_{3,12}^2 & 0
   \end{pmatrix},\quad
  {\mathbf G}
  =\begin{pmatrix}
     0       & g^{(12)} & g^{(13)} \\[0.1cm]
     g^{(21)} & 0       & g^{(23)} \\[0.1cm]
     g^{(31)} & g^{(32)} & 0
  \end{pmatrix}, \\[6ex]
  \displaystyle
  \delta\mathbf G
  =\begin{pmatrix}
     0 & \mbox{det}\,\Big(\nabla g^{(13)},\nabla g^{(32)}\Big)
       & \mbox{det}\,\Big(\nabla g^{(12)},\nabla g^{(23)}\Big) \\[0.3cm]
     \mbox{det}\,\Big(\nabla g^{(23)},\nabla g^{(31)}\Big)
       & 0
       & \mbox{det}\,\Big(\nabla g^{(21)},\nabla g^{(13)}\Big)\\[0.3cm]
     \mbox{det}\Big(\nabla g^{(32)},\nabla g^{(21)}\Big)
       & \mbox{det}\,\Big(\nabla g^{(31)},\nabla g^{(12)}\Big)  & 0
   \end{pmatrix}.
\end{array}
\end{equation*}
Comparing with (\ref{4.4}) gives the three equations
\begin{equation*}
\begin{array}{lll}
  \Delta g^{(12)}\negthickspace
  & = & \negthickspace\displaystyle
        g_v^{(13)}g_u^{(32)}-g_u^{(13)}g_v^{(32)}+S_{1,12}^2\,, \\[0.2cm]
  \Delta g^{(13)}\negthickspace
  & = & \negthickspace\displaystyle
        g_v^{(12)}g_u^{(23)}-g_u^{(12)}g_v^{(23)}+S_{1,12}^3\,, \\[0.2cm]
  \Delta g^{(23)}\negthickspace
  & = & \negthickspace\displaystyle
        g_v^{(21)}g_u^{(13)}-g_u^{(21)}g_v^{(13)}+S_{2,12}^3\,.
\end{array}
\end{equation*}
Now, if we set ${\mathcal G}:=(g^{(12)},g^{(13)},g^{(23)})$ and ${\mathcal S}:=(S_{1,12}^2,S_{1,12}^3,S_{2,12}^3),$ then
\begin{equation*}
  \Delta{\mathcal G}={\mathcal G}_u\times{\mathcal G}_v+\mathcal S\quad\mbox{in}\ B,\quad\mathcal G=0\quad\mbox{on}\ \partial B
\end{equation*}
with the usual vector product $\times$ in $\mathbb R^3.$ That means: \emph{$\mathcal G$ solves an inhomogeneous $H$-surface system with $H=\frac12$ and vanishes on the boundary.}
\end{itemize}
\subsection{The Grassmann-type vectors ${\mathcal G}$, $\delta\mathcal G$, and ${\mathcal S}$}
The last example gives rise to the definition of the following {\it vector of Grassmann type}
\begin{equation}\label{4.12}
  {\mathcal G}:=\big(g^{(\sigma\vartheta)}\big)_{1\le\sigma<\vartheta\le n}\in\mathbb R^N\,,\quad N:=\frac{n}{2}\,(n-1).
\end{equation}
In our examples, ${\mathcal G}$ works as follows:
\begin{equation*}
\begin{array}{lll}
  {\mathcal G}=g^{(12)}\in\mathbb R                             & \mbox{for}\ n=2, \\[0.1cm]
  {\mathcal G}=\big(g^{(12)},g^{(13)},g^{(23)}\big)\in\mathbb R^3 & \mbox{for}\ n=3.
\end{array}
\end{equation*}
Analogously, we define the Grassmann-type vectors
\begin{equation*}
  \delta\mathcal G:=\big(\delta g^{(\sigma\vartheta)}\big)_{1\le\sigma<\vartheta\le n}\in\mathbb R^N,\quad 
  \mathcal S:=\big(S_{\sigma,12}^\vartheta\big)_{1\le\sigma<\vartheta\le n}\in\mathbb R^N\,.
\end{equation*}
Then, the relations (\ref{4.4}) can be written as
\begin{equation}\label{4.15}
  \Delta\mathcal G=-\delta\mathcal G+\mathcal S\quad\mbox{in}\ B,\quad\mathcal G=0\quad\mbox{on}\ \partial B.
\end{equation}
From the definition of $\delta\mathcal G$, we immediately obtain the estimate 
\begin{equation}\label{4.16}
  |\Delta\mathcal G|\le c\,|\nabla\mathcal G|^2+|\mathcal S|\quad\mbox{in}\ B 
\end{equation}
with some constant $c>0$. 
\begin{remarks}\quad
\begin{itemize}
\item[1.]
The fact that $\Delta\mathcal G$ grows quadratically in $\nabla\mathcal G$ enables us to apply fundamental results on nonlinear elliptic systems due to E.\,Heinz \cite{Heinz_03}, \cite{Heinz_01} and F.\,Sauvigny \cite{Sauvigny_02}. And the special structure of $\delta\mathcal G$ allows us to utilize H.\,C.\,Wente's $L^\infty$-estimate \cite{Wente_01}, \cite{Topping_01}.
\vspace*{-0.6ex}
\item[2.]
For the homogeneous case $\mathcal S=0$, existence results for systems of the type (\ref{4.15}) can be found, e.g., in \cite{Heinz_02}, \cite{Wente_01} (for $n=3$), and \cite{Sauvigny_02} (for $n\ge 3$). In \cite{Takahashi_01} existence and multiplicity questions have been addressed in the inhomogeneous case (in our language: non-trivial bundle) for codimension $n=3$.
\vspace*{-0.6ex}
\item[3.]
The result in \cite{Sauvigny_02} can be extended to the inhomogeneous case: The boundary value problem (\ref{4.15}), (\ref{4.16}) has a solution ${\mathcal G}$, whenever ${\mathcal S}$ satisfies a smallness condition.
\vspace*{-0.6ex}
\item[4.]
Starting with a critical ONS $\mathcal N$, the mapping $\mathcal G=(g^{(\sigma\vartheta)})_{1\le\sigma<\vartheta\le n}$ from (\ref{4.1}), (\ref{4.2}) turns out to be a solution of (\ref{4.15}). Vice versa, solving (\ref{4.15}) for given $\mathcal S$ provides a first step towards the construction of a critical ONS $\mathcal N$.
\end{itemize}
\end{remarks}
\noindent
We plan to return to the questions in 3. and 4. in the future.
\subsection{A useful estimate}
Because the exact knowledge of the constant $c>0$ in (\ref{4.16}) will become important in section 6, we conclude the present section with the following
\begin{proposition}
It holds
\begin{equation}\label{4.17}
  |\Delta{\mathcal G}|
  \le\frac{\sqrt{n-2}}{2}\,|\nabla\mathcal G|^2+|\mathcal S|
  \quad\mbox{in}\ B.
\end{equation}
\end{proposition}
\begin{proof}From (\ref{4.15}) we know
\begin{equation}\label{4.18}
  |\Delta{\mathcal G}|\le |\delta\mathcal G|+|\mathcal S|\quad\mbox{in}\ B.
\end{equation}
It remains to estimate $|\delta\mathcal G|$ appropriately.
\begin{itemize}
\item[1.]
We begin with the inequality
\begin{equation}\label{4.19}
\begin{array}{rcl}
  |\delta\mathcal G|^2\negthickspace
  & = & \negthickspace\displaystyle
        \sum_{1\le\sigma<\vartheta\le n}
        \left\{\,
          \sum_{\omega=1}^n
          \det\big(\nabla g^{(\sigma\omega)},\nabla g^{(\omega\vartheta)}\big)
        \right\}^2\\[4ex]
  & \le & \negthickspace\displaystyle
          (n-2)
          \sum_{1\le\sigma<\vartheta\le n}
          \left\{\,
            \sum_{\omega=1}^n
            \det\big(\nabla g^{(\sigma\omega)},\nabla g^{(\omega\vartheta)}\big)^2
          \right\} \\[4ex]
  & = & \negthickspace\displaystyle
        (n-2)
        \sum_{1\le\sigma<\vartheta\le n}
        \Bigg\{\,
          \sum_{\omega<\sigma}
          \det\big(\nabla g^{(\omega\sigma)},\nabla g^{(\omega\vartheta)}\big)^2
          +\sum_{\sigma<\omega<\vartheta}
           \det\big(\nabla g^{(\sigma\omega)},\nabla g^{(\omega\vartheta)}\big)^2 \\[4ex]
  &  & \negthickspace\displaystyle
        \hspace{20ex}
        +\sum_{\vartheta<\omega}
         \det\big(\nabla g^{(\sigma\omega)},\nabla g^{(\vartheta\omega)}\big)^2\Bigg\}\,.
\end{array}
\end{equation}
Note that only derivatives of elements of $\mathcal G$ appear on the right hand side of (\ref{4.19}).
\item[2.]
Denote by $e_i=(0,\ldots,0,1,0,\ldots,0)\in\mathbb R^m$ the $i$-th standard basis vector. We recall the \emph{exterior wedge product} of two vectors $X,Y\in\mathbb R^m,$
\begin{equation*}
  X\wedge Y=\sum_{1\le i<j\le m}(-1)^{ij}(x^iy^j-x^jy^i)\,e_i\wedge e_j\,,
\end{equation*}
where $\{e_i\wedge e_j\}_{1\le i<j\le m}$ forms an orthonormal basis of the Euclidean space $\mathbb R^M$ for $M:=\frac{m}{2}(m-1)$ (see e.g. \cite{Heil_01}). Using the Lagrange's identity, we may estimate
\begin{equation}\label{4.21}
  |X\wedge Y|^2=|X|^2|Y|^2-(X\cdot Y^t)^2\le|X|^2|Y|^2\,.
\end{equation}
\item[3.]
Applying these settings to $\mathcal G$ with $m=N$ ($N$ from (\ref{4.12})), relation (\ref{4.19}) yields
\begin{equation}\label{4.22}
  |\delta{\mathcal G}|^2
  \le (n-2)|{\mathcal G}_u\wedge{\mathcal G}_v|^2\le (n-2)|{\mathcal G}_u|^2|{\mathcal G}_v|^2\,.
\end{equation}
Actually, ${\mathcal G}_u\wedge{\mathcal G}_v$ has more components than appear on the right hand side of (\ref{4.19}).
Combining (\ref{4.22}) with (\ref{4.18}) gives
\begin{equation*}
  |\Delta\mathcal G|
  \le\sqrt{n-2}\,|\mathcal G_u||\mathcal G_v|+|\mathcal S|
  \le\frac{\sqrt{n-2}}{2}\,|\nabla\mathcal G|^2+|\mathcal S|,
\end{equation*}
which proves the statement.\vspace*{-5ex}
\end{itemize}
\end{proof}
\section{Immersions with flat normal bundle}
\label{section5}
Assuming that the normal bundle of a given immersion $X$ is flat, we prove that any ONS ${\mathcal N},$ which is critical for the functional of total torsion, must be free of torsion (then ${\mathcal N}$ is called a {\it parallel} section).
\setcounter{equation}{0}
\subsection{Immersions with flat normal bundle}
\begin{definition}
The immersion $X$ has flat normal bundle ${\mathfrak S}\equiv 0$ iff $S_{\sigma,ij}^\vartheta\equiv 0$ in $B$ for all $i,j=1,2$ and $\sigma,\vartheta=1,\ldots,n.$
\end{definition}
\subsection{A lemma on an auxiliary function}
For the proof we need the following
\begin{lemma}
Let the immersion $X$ with flat normal bundle ${\mathfrak S}\equiv 0$ together with a critical ONS ${\mathcal N}$ be given. Then the function
\begin{equation*}
  f(w):=\mathcal G_w(w)\cdot\mathcal G_w^t(w)
\end{equation*}
vanishes identically in $B.$
\end{lemma}
\begin{remark}
Here, we use Wirtinger's calculus
\begin{equation*}
  \varphi_w:=\varphi_u-i\varphi_v\,,\quad
  \varphi_{\overline w}:=\varphi_u+i\varphi_v\,,\quad
  w=u+iv,
\end{equation*}
for a complex-valued function $\varphi=\varphi(w).$ Take note of the relation $\varphi_{w\overline w}=\Delta\varphi$.
\end{remark}
\begin{proof}[Proof of the lemma]
We will prove that $f$ solves the boundary value problem
\begin{equation*}
  f_{\overline w}=0\quad\mbox{in}\ B,\quad \mbox{Im}(w^2f)=0\quad\mbox{on}\ \partial B.
\end{equation*}
Then, the analytic function $g(w):=w^2f(w)$ has vanishing imaginary part, the Cauchy-Rie\-mann equations imply $g(w)\equiv c\in\mathbb R,$ and the assertion follows from $g(0)=0.$
\begin{itemize}
\item[1.]
In order to deduce the stated boundary condition, recall that $g^{(\sigma\vartheta)}=0$ on $\partial B$. Thus, all tangential derivatives vanish identically:
\begin{equation*}
  -vg_u^{(\sigma\vartheta)}+ug_v^{(\sigma\vartheta)}=-\mbox{Im}(wg_w^{(\sigma\vartheta)})=0\quad\mbox{on}\ \partial B
\end{equation*}
for all $\sigma,\vartheta=1,\ldots,n.$ The statement follows from
\begin{equation*}
\hspace*{-2ex}
\begin{array}{lll}
  \mbox{Im}\,(w^2f)\negthickspace
  & = & \negthickspace\displaystyle
        \mbox{Im}\,\Big(w^2\,\mathcal G_w\cdot\mathcal G^t_w\Big)
        \,=\,\mbox{Im}\,
             \bigg\{
               w^2\sum_{1\le\sigma<\vartheta\le n}g_w^{(\sigma\vartheta)}g_w^{(\sigma\vartheta)}
             \bigg\} \\[4ex]
  & = & \negthickspace\displaystyle
        \sum_{1\le\sigma<\vartheta\le n}\!
             \mbox{Im}\,\Big\{\big(wg_w^{(\sigma\vartheta)}\big)\big(wg_w^{(\sigma\vartheta)}\big)\Big\}
        \,=\,2\!\sum_{1\le\sigma<\vartheta\le n}\!
         \mbox{Re}\,\big(wg_w^{(\sigma\vartheta)}\big)\,
         \mbox{Im}\,\big(wg_w^{(\sigma\vartheta)}\big)\,=\,0
\end{array}
\end{equation*}
\item[2.]
Finally, we show the analyticity of $f$  with the aid of (\ref{4.4}): Interchanging indices cyclically yields
\begin{equation*}
\begin{array}{lll}
  f_{\overline w}\negthickspace
  & = & \negthickspace\displaystyle
        2\,\mathcal G_w\cdot\mathcal G^t_{w\overline w}
        \,=\,2\sum_{1\le\sigma<\vartheta\le n}
             g_w^{(\sigma\vartheta)}g_{w\overline w}^{(\sigma\vartheta)}
		\,=\,\sum_{\sigma,\vartheta=1}^n
			 g_w^{(\sigma\vartheta)}\Delta g^{(\sigma\vartheta)}\\[4ex]
  & = & \negthickspace\displaystyle
        \sum_{\sigma,\vartheta,\omega=1}^n
        \Big\{
          g_v^{(\sigma\omega)}g_u^{(\omega\vartheta)}g_u^{(\sigma\vartheta)}
          -g_u^{(\sigma\omega)}g_v^{(\omega\vartheta)}g_u^{(\sigma\vartheta)}
        \Big\} \\[4ex]
  &   & \negthickspace\displaystyle
        -\,i\sum_{\sigma,\vartheta,\omega=1}^n
            \Big\{
              g_v^{(\sigma\omega)}g_u^{(\omega\vartheta)}g_v^{(\sigma\vartheta)}
              -g_u^{(\sigma\omega)}g_v^{(\omega\vartheta)}g_v^{(\sigma\vartheta)}
            \Big\} \\[4ex]
  & = & \negthickspace\displaystyle
        \sum_{\sigma,\vartheta,\omega=1}^n
        \Big\{
          g_v^{(\omega\vartheta)}g_u^{(\vartheta\sigma)}g_u^{(\omega\sigma)}
          -g_u^{(\sigma\omega)}g_v^{(\omega\vartheta)}g_u^{(\sigma\vartheta)}
        \Big\} \\[4ex]
  &   & \negthickspace\displaystyle
        -\,i\sum_{\sigma,\vartheta,\omega=1}^n
            \Big\{
              g_v^{(\vartheta\sigma)}g_u^{(\sigma\omega)}g_v^{(\vartheta\omega)}
              -g_u^{(\sigma\omega)}g_v^{(\omega\vartheta)}g_v^{(\sigma\vartheta)}
            \Big\},
\end{array}
\end{equation*}
which shows $f_{\overline w}=0.$ The proof is complete.\vspace{-5ex}
\end{itemize}
\end{proof}
\subsection{Torsion-free ONS for flat normal bundles}
Our first theorem concerns the torsion of critical ONS ${\mathcal N}$ in the case of flat normal bundles.
\begin{theorem}
Let $X\in C^4(B,\mathbb R^{n+2})$ be an immersion with flat normal bundle ${\mathfrak S}\equiv 0.$ Then, for any critical ONS ${\mathcal N},$ the torsions $T_{\sigma,i}^\vartheta$, $i=1,2$, $\sigma,\vartheta=1,\ldots,n$, vanish identically in $B$.
\end{theorem}
\begin{proof}
Consider the Grassmann-type vector ${\mathcal G}\in C^2(B,\mathbb R^N)$ from (\ref{4.12}). Because ${\mathcal G}_w\cdot{\mathcal G}_w^t$ vanishes by the above lemma, there hold
\begin{equation*}
  |{\mathcal G}_u|=|{\mathcal G}_v|,\quad
  {\mathcal G}_u\cdot{\mathcal G}_v^t=0
  \quad\mbox{in}\ B.
\end{equation*}
This means that ${\mathcal G}$ is a conformally parametrized solution of
\begin{equation*}
  \Delta{\mathcal G}=-\,\delta{\mathcal G}\quad\mbox{in}\ B,\quad
  {\mathcal G}=0\quad\mbox{on}\ \partial B;
\end{equation*}
see (\ref{4.17}) with ${\mathcal S}=0$. According to the growth condition $|\delta\mathcal G|\le c|\nabla\mathcal G|^2$, the arguments in \cite{Heinz_01} apply: Assume ${\mathcal G}\not\equiv\mbox{const}$ in $B$. Then, the asymptotic expansion stated in the Satz of \cite{Heinz_01} implies that boundary branch points $w_0\in\partial B$ with ${\mathcal G}_u(w_0)={\mathcal G}_v(w_0)=0$ are isolated. But this contradicts our boundary condition ${\mathcal G}|_{\partial B}=0$ from (\ref{4.15}). Thus, ${\mathcal G}(w)\equiv\mbox{const}=0$ and, finally, the definition (\ref{4.1}) implies $T_{\sigma,i}^\vartheta\equiv0$ in $B.$
\end{proof}
\noindent
As an immediate consequence, we obtain the
\begin{corollary}
If the immersion $X\in C^4(B,\mathbb R^{n+2})$ has flat normal bundle ${\mathfrak S}\equiv 0,$ then any critical ONS ${\mathcal N}$ is optimal w.r.t.~${\cal T}_X,$ i.e. ${\cal T}_X({\mathcal N})=0.$
\end{corollary}
\begin{remark}
The case $n=3$ ($X$ is immersed in $\mathbb R^5$) is covered by Wente's result in \cite{Wente_02}.
\end{remark}
\setcounter{equation}{0}
\section{The case of non-flat normal bundle}
\label{section6}
If the normal bundle has non-vanishing curvature $\mathcal S\not\equiv0$ it is desirable to have both, lower and upper bounds, at least for the total torsion of a critical orthonormal normal section.\\[1ex]
In paragraph 6.1, we will prove such a lower bound for the functional ${\mathcal T}_X.$ In the remaining paragraphs we establish an upper bound for  the torsion coefficients combining a gradient estimate due to E.\,Heinz with H.\,C.\,Wente's $L^\infty$-estimate.
\subsection{A lower bound for the total torsion}
We write $\|Z\|_{p,\varrho}$, $p\in[1,+\infty]$, $\varrho\in[0,1]$, for the $L^p(B_\varrho(0))$-norm of a continuous mapping $Z\colon B_\varrho(0)\to\mathbb R^d,$ $d\in\mathbb N.$ In addition, we abbreviate $\|Z\|_p:=\|Z\|_{p,1}$.
\begin{theorem}
Let $X\in C^4(B,\mathbb R^{n+2}),$ $n\ge 2,$ be an immersion and ${\mathcal N}$ a critical ONS of its normal bundle with curvature ${\mathcal S}\not\equiv 0.$
\begin{itemize}
\item[(I)]
If ${\mathcal S}\not=\mbox{\rm const},$ then it holds
\begin{equation}\label{6.1}
  {\mathcal T}_X({\mathcal N})\ge \bigg(\sqrt{n-2}\,\|\mathcal S\|_\infty+\frac{\|\mathcal S\|_{2}^{2}}{(1-\varrho)^2\|\mathcal S\|_{2,\varrho}^2}
  +\frac{2\|\nabla\mathcal S\|_2^{2}}{\|\mathcal S\|_{2,\varrho}^2}\bigg)^{-1}\|\mathcal S\|_{2,\varrho}^2>0
\end{equation}
with $\varrho=\varrho(\mathcal S)\in(0,1)$ chosen as in (\ref{6.3}).
\item[(II)]
If ${\mathcal S}=\mbox{\rm const}\not=0,$ then we have 
\begin{equation}\label{6.2}
  {\mathcal T}_X({\mathcal N})
  \ge \frac{1}{2}\,\frac{\pi|\mathcal S|^2}{\sqrt{n-2}\,|\mathcal S|+16}.
\end{equation}
\end{itemize}
\end{theorem}
\begin{proof}
\begin{itemize}
\item[1.]
We start with (I): Because of ${\mathcal S}\not=\mbox{const},$ there exists $\varrho=\varrho(\mathcal S)\in(0,1)$ such that
\begin{equation}\label{6.3}
  \|\mathcal S\|_{2,\varrho}
  =\left(\ \,
     \int\hspace{-0.25cm}\int\limits_{\hspace{-0.3cm}B_\varrho(0)}|{\mathcal S}|^2\,dudv
   \right)^{\frac12}>0.
\end{equation}
We choose a test function $\eta\in C^0(B,\mathbb R)\cap\mathring{H}_2^1(B,\mathbb R)$ with the properties
\begin{equation}\label{6.4}
  \eta\in[0,1]\quad\mbox{in}\ B,\quad
  \eta=1\quad\mbox{in}\ B_\varrho\,,\quad
  |\nabla\eta|\le\frac{1}{1-\varrho}\quad\mbox{in}\ B.
\end{equation}
Multiplying $\Delta{\mathcal G}=-\delta{\mathcal G}+{\mathcal S}$ from (\ref{4.15}) by $(\eta{\mathcal S})$ and integrating by parts yields
\begin{equation*}
  \int\hspace{-0.25cm}\int\limits_{\hspace{-0.3cm}B}\nabla\mathcal G\cdot\nabla(\eta\mathcal S)^t=\int\hspace{-0.25cm}\int\limits_{\hspace{-0.3cm}B}
  \eta\,\delta\mathcal G\cdot\mathcal S^t-\int\hspace{-0.25cm}\int\limits_{\hspace{-0.3cm}B}\eta\,|\mathcal S|^2
\end{equation*}
(we omit $dudv$). Taking (\ref{4.22}) into account, we can now estimate as follows:
\begin{equation}\label{6.6}
\begin{array}{lll}
  \displaystyle
  \int\hspace{-0.25cm}\int\limits_{\hspace{-0.3cm}B_\varrho}|{\mathcal S}|^2\negthickspace
  & \le & \negthickspace\displaystyle
          \int\hspace{-0.25cm}\int\limits_{\hspace{-0.3cm}B}
          \eta\,|{\mathcal S}|^2
          \,\le\,\int\hspace{-0.25cm}\int\limits_{\hspace{-0.3cm}B}
                \eta\,\big|\delta{\mathcal G}\cdot{\mathcal S}^t\big|
				+\int\hspace{-0.25cm}\int\limits_{\hspace{-0.3cm}B}
               \big|\nabla{\mathcal G}\cdot\nabla(\eta{\mathcal S})^t\big|\\[5ex]
  & \le & \negthickspace\displaystyle
          \|{\mathcal S}\|_\infty
          \int\hspace{-0.25cm}\int\limits_{\hspace{-0.3cm}B}\eta\,|\delta{\mathcal G}|
          +\int\hspace{-0.25cm}\int\limits_{\hspace{-0.3cm}B}
           |\nabla\eta|\,|\mathcal S|\,|\nabla{\mathcal G}|
          +\int\hspace{-0.25cm}\int\limits_{\hspace{-0.3cm}B}
           \eta\,|\nabla\mathcal S|\,|\nabla\mathcal G| \\[5ex]
  & \le & \negthickspace\displaystyle
          \frac{\sqrt{n-2}}{2}\,\|{\mathcal S}\|_\infty
          \int\hspace{-0.25cm}\int\limits_{\hspace{-0.3cm}B}
          |\nabla{\mathcal G}|^2
          +\frac\varepsilon2
           \int\hspace{-0.25cm}\int\limits_{\hspace{-0.3cm}B}
		   |\mathcal S|^2+\frac1{2\varepsilon(1-\varrho)^2}	
		   \int\hspace{-0.25cm}\int\limits_{\hspace{-0.3cm}B}
           |\nabla{\mathcal G}|^2 \\[5ex]
  &&      \negthickspace\displaystyle
          +\frac\delta2
           \int\hspace{-0.25cm}\int\limits_{\hspace{-0.3cm}B}
           |\nabla\mathcal S|^2
          +\frac1{2\delta}
		   \int\hspace{-0.25cm}\int\limits_{\hspace{-0.3cm}B}
		   |\nabla\mathcal G|^2
\end{array}
\end{equation}
with arbitrary numbers $\varepsilon,\delta>0$. Let us write (\ref{6.6}) as
\begin{equation}\label{6.7}
  \|\mathcal S\|_{2,\varrho}^2
  \le\left(
       \frac{\sqrt{n-2}}{2}\,\|\mathcal S\|_\infty
       +\frac{1}{2\varepsilon(1-\varrho)^2}
       +\frac1{2\delta}
      \right)\|\nabla\mathcal G\|_2^2
      +\frac{\varepsilon}{2}\,\|\mathcal S\|_2^2
      +\frac{\delta}{2}\,\|\nabla\mathcal S\|_2^2\,.
\end{equation}
\item[2.]
According to (\ref{6.3}), the choice $\varepsilon=\|\mathcal S\|_2^{-2}\|\mathcal S\|_{2,\varrho}^2>0$ is admissible in (\ref{6.7}), and we infer
\begin{equation}\label{6.8}
  \|\mathcal S\|_{2,\varrho}^2
  \le\left(
       \sqrt{n-2}\,\|\mathcal S\|_\infty
       +\frac{\|\mathcal S\|_{2}^{2}}{(1-\varrho)^2\|\mathcal S\|_{2,\varrho}^2}
       +\frac{1}{\delta}
     \right)\|\nabla\mathcal G\|_2^2
     +\delta\|\nabla\mathcal S\|_2^2\,.
\end{equation}
And since $\mathcal S\not=\mbox{\rm const},$ we can choose $\delta=\frac12\|\nabla\mathcal S\|_2^{-2}\|\mathcal S\|_{2,\varrho}^2$ in (\ref{6.8}), which implies
\begin{equation*}
  \|\mathcal S\|_{2,\varrho}^2
  \le 2\left(
         \sqrt{n-2}\,\|\mathcal S\|_\infty
         +\frac{\|\mathcal S\|_{2}^{2}}{(1-\varrho)^2\|\mathcal S\|_{2,\varrho}^2}
         +\frac{2\|\nabla\mathcal S\|_2^{2}}{\|\mathcal S\|_{2,\varrho}^2}
       \right)\|\nabla\mathcal G\|_2^2\,.
\end{equation*}
Having $\mathcal T_X(\mathcal N)=2\|\nabla\mathcal G\|_2^2$ in mind, we arrive at (\ref{6.1}).
\item[3.]
In the case ${\mathcal S}=\mbox{\rm const}\not=0$ we choose $\varrho=\frac{1}{2}$ in (\ref{6.4}). Starting as in (\ref{6.6}), we then obtain
\begin{equation*}
\begin{array}{lll}
  \displaystyle
  \frac\pi4|\mathcal S|^2\negthickspace
  & = & \negthickspace\displaystyle
          \int\hspace{-0.25cm}\int\limits_{\hspace{-0.3cm}B_{\frac12}}|{\mathcal S}|^2
          \,\le\,|\mathcal S|\int\hspace{-0.25cm}\int\limits_{\hspace{-0.3cm}B}
           |\delta{\mathcal G}|
		  +\int\hspace{-0.25cm}\int\limits_{\hspace{-0.3cm}B}
           |\nabla\eta|\,|\mathcal S|\,|\nabla{\mathcal G}| \\[5ex]
  & \le & \negthickspace\displaystyle
          \frac{\sqrt{n-2}}{2}\,|{\mathcal S}|
          \int\hspace{-0.25cm}\int\limits_{\hspace{-0.3cm}B}
          |\nabla{\mathcal G}|^2
          +\frac\varepsilon2\pi|\mathcal S|^2
           +\frac2\varepsilon	
		   \int\hspace{-0.25cm}\int\limits_{\hspace{-0.3cm}B}
           |\nabla{\mathcal G}|^2\,.
\end{array}
\end{equation*}
With $\varepsilon=\frac14$ it follows that
\begin{equation*}
  \frac{\pi}{8}\,|\mathcal S|^2
  \le\left(
       \frac{\sqrt{n-2}}{2}\,|{\mathcal S}|+8
     \right)\|\nabla\mathcal G\|_2^2\,.
\end{equation*}
This implies the estimate (\ref{6.2}).\vspace*{-3ex}
\end{itemize}
\end{proof}
\begin{remarks}\quad
\begin{itemize}
\item[1.]
For {\it small} solutions ${\mathcal G}$ with $\|\mathcal G\|_\infty<\frac2{\sqrt{n-2}}$ it is quite easy to derive also an upper bound for the total torsion: Multiplying (\ref{4.15}) by ${\mathcal G}$ and integrating by parts yields
  $$\mathcal T_X(\mathcal N)=2\|\nabla\mathcal G\|_2^2\le \frac{4\|\mathcal G\|_\infty\|\mathcal S\|_1}{2-\sqrt{n-2}\,
  \|{\mathcal G}\|_\infty}\,.$$
Such a small solution can be constructed via the arguments in \cite{Sauvigny_02}; see remark 3 in subsection 4.4. Let us emphasize here again that the case $n=2$ is much easier to handle: The classical maximum principle controls $\|g^{(12)}\|_\infty$ by $\|S_{1,12}^2\|_\infty,$ and no smallness condition is needed to bound the total torsion; see subsection 4.3.
\item[2.]
In \cite{Takahashi_01} F.\,Takahashi translated the system (\ref{4.15}) for $n=3$ into a variational problem. Then he was able to derive lower and upper bounds for the quantity $\|\nabla\mathcal G\|_2$ of a minimizer in the corresponding Nehari manifold, whenever $\mathcal S$ is sufficiently small (in the $H^{-1}(B)$-norm); we refer to \cite{Takahashi_01} for the details.
\end{itemize}
\end{remarks}
\subsection{An $L^\infty$-bound for ${\mathcal G}$}
\begin{proposition}
For a critical ONS ${\mathcal N},$ the Grassmann-type vector ${\mathcal G}$ from (\ref{4.12}) satisfies
\begin{equation}\label{6.13}
  \|\mathcal G\|_{\infty}
  \le\frac{n-2}{2\pi}\,\|\nabla \mathcal G\|_2^2+\frac{1}{4}\,\sqrt{\frac{n(n-1)}2}\,\|\mathcal S\|_\infty\,.
\end{equation}
\end{proposition}
\begin{proof}
\begin{itemize}
\item[1.]
For $1\le \sigma<\vartheta\le n$ and $\omega\in\{1,\ldots,n\}$ with $\omega\not\in\{\sigma,\vartheta\}$, we define the functions $y^{(\sigma\vartheta\omega)}$ as the unique solutions of
\begin{equation*}
  \Delta y^{(\sigma\vartheta\omega)}=-\det\big(\nabla g^{(\sigma\omega)},\nabla g^{(\omega\vartheta)}\big)\quad\mbox{in}\ B,\quad
  y^{(\sigma\vartheta\omega)}=0\quad\mbox{on}\ \partial B.
\end{equation*}
Wente's $L^\infty$-estimate (compare e.g.~\cite{Wente_01}, \cite{Topping_01}) then yields the \emph{optimal} inequalities
\begin{equation}\label{6.15}
  \|y^{(\sigma\vartheta\omega)}\|_\infty\le\frac 1{4\pi}\Big(\|\nabla g^{(\sigma\omega)}\|_2^2+\|\nabla g^{(\omega\vartheta)}\|_2^2\Big),
  \quad(\sigma,\vartheta)\in U_n,\quad\omega\not\in\{\sigma,\vartheta\}.
\end{equation}
In addition, we introduce the Grassmann-type vector $\mathcal Z=(z^{(\sigma\vartheta)})_{1\le\sigma<\vartheta\le n}$ as the unique solution of
\begin{equation*}
  \Delta \mathcal Z=\mathcal S\quad\mbox{in}\ B,\quad
  \mathcal Z=0\quad\mbox{on}\ \partial B.
\end{equation*}
We use Poisson's representation formula and estimate as follows:
\begin{equation}\label{6.17}
\begin{array}{lll}
  \displaystyle
  |\mathcal Z(w)|\negthickspace
  &  =  & \negthickspace\displaystyle
          \bigg|
            \int\hspace{-0.25cm}\int\limits_{\hspace{-0.3cm}B}
              \phi(\zeta;w)\mathcal S(\zeta)\,d\xi d\eta
          \bigg|
          \,\le\,\sqrt N
                 \int\hspace{-0.25cm}\int\limits_{\hspace{-0.3cm}B}
                 |\phi(\zeta;w)||\mathcal S(\zeta)|\,d\xi d\eta \\[4ex]
  & \le & \negthickspace\displaystyle
          \sqrt N\,\|{\mathcal S}\|_\infty
          \int\hspace{-0.25cm}\int\limits_{\hspace{-0.3cm}B}
          |\phi(\zeta;w)|\,d\xi d\eta
\end{array}
\end{equation}
with the non-positive Green's function 
\begin{equation}\label{6.18}
  \phi(\zeta;w):=\frac1{2\pi}\log\Big|\frac{\zeta-w}{1-\overline w\zeta}\Big|,\quad\zeta\not= w,
\end{equation}
for $\Delta$ in $B;$ $\zeta=(\xi,\eta)$.
Because $\psi(w)=\frac{|w|^2-1}{4}$ solves $\Delta\psi=1$ in $B,$ $\psi=0$ on $\partial B,$ Poisson's formula yields
\begin{equation*}
  \int\hspace{-0.25cm}\int\limits_{\hspace{-0.3cm}B}
  |\phi(\zeta;w)|\,d\xi d\eta
  =\frac{1-|w|^2}{4}
  \le\frac{1}{4}\,,
\end{equation*}
which enables us to continue (\ref{6.17}) to get
\begin{equation}\label{6.19}
  \|\mathcal Z\|_\infty\le\frac{\sqrt N}4\|\mathcal S\|_\infty\,.
\end{equation}
\item[2.]
Next, we note the identity 
\begin{equation*}
  g^{(\sigma\vartheta)}=\sum_{\omega\not\in\{\sigma,\vartheta\}}y^{(\sigma\vartheta\omega)}+z^{(\sigma\vartheta)},\quad(\sigma,\vartheta)\in U_n.
\end{equation*}
Applying now the estimates (\ref{6.15}) and (\ref{6.19}), we arrive at
\begin{equation*}
\begin{array}{rcl}
  \|\mathcal G\|_\infty\negthickspace
  & \le & \negthickspace\displaystyle
          \sum_{\sigma<\vartheta}
          \sum_{\omega\not\in\{\sigma,\vartheta\}}
          \|y^{(\sigma\vartheta\omega)}\|_\infty+\|\mathcal Z\|_\infty\\[4ex]
  & \le & \negthickspace\displaystyle
          \frac1{4\pi}
          \sum_{\sigma<\vartheta}
          \sum_{\omega\not\in\{\sigma,\vartheta\}}
          \Big(
            \|\nabla g^{(\sigma\omega)}\|_2^2+\|\nabla g^{(\omega\vartheta)}\|_2^2
          \Big)
          +\frac{\sqrt N}{4}\,\|\mathcal S\|_\infty\\[4ex]
  & = & \negthickspace\displaystyle
          \frac1{4\pi}\,
          \bigg\{
            \sum_{\omega<\sigma<\vartheta}
            \Big(
              \|\nabla g^{(\omega\sigma)}\|_2^2+\|\nabla g^{(\omega\vartheta)}\|_2^2
            \Big)
            +\sum_{\sigma<\omega<\vartheta}
             \Big(
               \|\nabla g^{(\sigma\omega)}\|_2^2+\|\nabla g^{(\omega\vartheta)}\|_2^2
             \Big)\\[4ex]
  &     & \negthickspace\displaystyle
          \hspace*{5.2ex}
            +\sum_{\sigma<\vartheta<\omega}
             \Big(
               \|\nabla g^{(\sigma\omega)}\|_2^2+\|\nabla g^{(\vartheta\omega)}\|_2^2
             \Big)
          \bigg\}
          +\frac{\sqrt N}{4}\,\|\mathcal S\|_\infty\\[4ex]
  &  =  & \negthickspace\displaystyle
          \frac1{4\pi}\,
          \bigg\{
            \sum_{\sigma<\vartheta<\omega}
            \|\nabla g^{(\sigma\vartheta)}\|_2^2
            +\sum_{\sigma<\omega<\vartheta}
             \|\nabla g^{(\sigma\vartheta)}\|_2^2
            +\sum_{\sigma<\vartheta<\omega}
             \|\nabla g^{(\sigma\vartheta)}\|_2^2\\[4ex]
  &     & \negthickspace\displaystyle
          \hspace*{5.2ex}
            +\!\sum_{\omega<\sigma<\vartheta}\!
             \|\nabla g^{(\sigma\vartheta)}\|_2^2
            +\!\sum_{\sigma<\omega<\vartheta}\!
             \|\nabla g^{(\sigma\vartheta)}\|_2^2
            +\!\sum_{\omega<\sigma<\vartheta}\!
             \|\nabla g^{(\sigma\vartheta)}\|_2^2
          \bigg\}
          +\frac{\sqrt N}{4}\,\|\mathcal S\|_\infty\\[4ex]
  &  =  & \negthickspace\displaystyle
          \frac1{2\pi}
          \sum_{\sigma<\vartheta}
          \sum_{\omega\not\in\{\sigma,\vartheta\}}
          \|\nabla g^{(\sigma\vartheta)}\|_2^2
          +\frac{\sqrt N}{4}\,\|\mathcal S\|_\infty\\[4ex]
  &  =  & \negthickspace\displaystyle
          \frac{n-2}{2\pi}\,\|\nabla\mathcal G\|_2^2
          +\frac{1}{4}\,\sqrt{\frac{n(n-1)}2}\,\|\mathcal S\|_\infty,
\end{array}
\end{equation*}
as asserted.
\end{itemize}
\end{proof}
\subsection{An alternative estimate for $\|{\mathcal G}\|_\infty$}
For large codimension $n$, the estimate (\ref{6.13}) is somewhat unsatisfactory. Alternatively, we will show the inequality
\begin{equation}\label{6.22}
  |z^{(\sigma\vartheta)}(w)|\le\sqrt{\frac{2}{\pi}}\,\|S_{\sigma,12}^\vartheta\|_2
  \quad\mbox{in}\ B\quad\mbox{for all}\ (\sigma,\vartheta)\in U_n
\end{equation}
in the present paragraph.
\goodbreak\noindent
Then we calculate
\begin{equation*}
  \|\mathcal Z\|_\infty
  =\sup_B\sqrt{\sum_{\sigma<\vartheta}|z^{(\sigma\vartheta)}(w)|^2}
  \le\sqrt{\frac2\pi}\,\sqrt{\sum_{\sigma<\vartheta}\|S_{\sigma,12}^\vartheta\|_2^2}
  =\sqrt{\frac2\pi}\,\|\mathcal S\|_2
  \le\sqrt2\,\|\mathcal S\|_\infty\,,
\end{equation*}
and this estimate instead of (\ref{6.19}) will lead us to a smaller upper bound for $\|{\mathcal G}\|_\infty$ at least for large codimensions $n.$\\[1ex]
In order to prove (\ref{6.22}), we use the H\"older and the Sobolev inequality and compute
\begin{equation}\label{6.24}
  |z^{(\sigma\vartheta)}(w)|
  \le\|\phi(\cdot\,;w)\|_2\|S_{\sigma,12}^\vartheta\|_2
  \le\frac1{2\sqrt\pi}\|\nabla_\zeta\phi(\cdot\,;w)\|_1\|S_{\sigma,12}^\vartheta\|_2\,.
\end{equation}
For the optimal constant $\frac1{2\sqrt\pi}$ in the Sobolev inequality we refer to \cite{Gilbarg_Trudinger_83} section 7.7 and the references therein. In (\ref{6.24}), $\phi=\phi(\zeta;w)$ denotes again Green's function (\ref{6.18}) for $\Delta$ in $B,$ which satisfies $\phi(\cdot\,;w)\in\mathring{H}^1_1(B)$ for any $w\in\mathring{B}$ as well as
\begin{equation*}
  \phi_\zeta(\zeta;w)
  \equiv\phi_\xi(\zeta;w)-i\phi_\eta(\zeta;w)
  =\frac{1}{2\pi}\,
   \overline{
     \left(
       \frac{\zeta-w}{|\zeta-w|^2}+w\frac{1-\overline w\zeta}{|1-\overline w\zeta|^2}
       \right)},
  \quad w\not=\zeta.
\end{equation*}
A straightforward calculation shows
\begin{equation*}
  |\nabla_\zeta\phi(\zeta;w)|
  \equiv|\phi_\zeta(\zeta;w)|
  =\frac1{2\pi}\frac{1-|w|^2}{|\zeta-w|\,|1-\overline w\zeta|}\le\frac1{2\pi}\frac{1+|w|}{|\zeta-w|}
  \le\frac1\pi\frac1{|\zeta-w|},\quad\zeta\not=w.
\end{equation*}
And since the right hand side in the inequality
\begin{equation*}
  \int\hspace{-0.25cm}\int\limits_{\hspace{-0.3cm}B}
  |\nabla_\zeta\phi(\zeta;w)|\,d\xi d\eta
  \le\frac{1}{\pi}\ \,
     \int\hspace*{-0.25cm}\int\limits_{\hspace*{-0.4cm}B_\delta(w)}
     \frac1{|\zeta-w|}\,d\xi\,d\eta
  +\frac{1}{\pi}\ \ \,
   \int\hspace*{-0.4cm}\int\limits_{\hspace*{-0.45cm}B\setminus B_\delta(w)}
   \frac1{|\zeta-w|}\,d\xi\,d\eta\le 2\delta+\frac1\delta
\end{equation*}
becomes minimal for $\delta=\frac1{\sqrt2}$, we arrive at (\ref{6.22}). Instead of (\ref{6.13}), we thus have the
\begin{proposition}
For a critical ONS ${\mathcal N},$ the Grassmann-type vector ${\mathcal G}$ from (\ref{4.15}) satisfies
\begin{equation}\label{6.28}
  \|\mathcal G\|_{\infty}\le\frac{n-2}{2\pi}\|\nabla \mathcal G\|_2^2+\sqrt2\,\|\mathcal S\|_\infty.
\end{equation}
\end{proposition}
\noindent
Note that (\ref{6.28}) provides a better bound than (\ref{6.13}) only in the case $n\ge 9.$
\subsection{A pointwise upper bound for the torsion coefficients}
We are now in the position to prove our third main result for immersions with non-flat normal bundle:
\begin{theorem}
Let $X\in C^4(B,\mathbb R^{n+2})$, $n\ge3$, be an immersion and $\mathcal N$ a critical ONS of its normal bundle. Assume that the smallness condition
\begin{equation}\label{6.29}
  \frac{\sqrt{n-2}}2
  \left(
    \frac{n-2}{4\pi}\,\mathcal T_X(\mathcal N)+\gamma(n)\|\mathcal S\|_\infty
  \right)<1
\end{equation}
is satisfied with $\gamma(n):=\min\{\frac14\sqrt{\frac{n(n-1)}2},\sqrt2\}$. Then, the torsion coefficients of ${\mathcal N}$ can be estimated by means of
\begin{equation}\label{6.30}
  \|T_{\sigma,i}^\vartheta\|_\infty\le c,\quad i=1,2,\ (\sigma,\vartheta)\in U_n,
\end{equation}
with a nonnegative constant $c=c(n,\|\mathcal S\|_\infty,\mathcal T_X(\mathcal N))<+\infty$.
\end{theorem}
\begin{remark}
For codimension $n=2$, the estimate (\ref{6.30}) can be proved with $c=c(\|\mathcal S\|_\infty)$ and without presuming a smallness condition (\ref{6.29}), as already indicated in subsection 4.3. Again we refer to \cite{Froehlich_Mueller_01} for a slight generalization of that result.

\noindent For $n\ge 3$, it remains open if it is possible to prove global pointwise estimates for the torsion coefficients without the smallness condition (\ref{6.29}) and without the knowledge of ${\mathcal T}_X$.
\end{remark}
\begin{proof}[Proof of the theorem]
According to (\ref{4.17}), (\ref{6.13}) resp.~(\ref{6.28}), the Grassmann-type vector $\mathcal G=(g^{(\sigma\vartheta)})_{\sigma<\vartheta}$ solves the system
\begin{equation*}
\begin{array}{l}
  |\Delta\mathcal G|\le a|\nabla\mathcal G|^2+b\quad\mbox{in}\ B,\\[1ex]
  \mathcal G=0\quad\mbox{on}\ \partial B,\\[1ex]
  \|\mathcal G\|_\infty\le M,
\end{array}
\end{equation*}
where the appearing constants are defined by
\begin{equation*}
  a:=\frac{\sqrt{n-2}}2,\quad b:=\|\mathcal S\|_\infty,\quad M:=\frac{n-2}{2\pi}\|\nabla\mathcal G\|_2^2+\gamma(n)\|\mathcal S\|_\infty.
\end{equation*}
The smallness condition (\ref{6.29}) assures $aM<1$ due to $\mathcal T_X(\mathcal N)=2\|\nabla\mathcal G\|_2^2$. Consequently, we can apply E.\,Heinz's global gradient estimate Theorem\,1 in \cite{Sauvigny_02} Chap.~XII, \S\,3, obtaining $\|\nabla\mathcal G\|_\infty\le c$. This in turn yields the desired estimate (\ref{6.30}), according to (\ref{4.1}) and (\ref{4.12}).
\end{proof}

\vspace*{10ex}
\noindent
Steffen Fr\"ohlich\\
Freie Universit\"at Berlin\\ 
Fachbereich Mathematik und Informatik\\ 
Arnimallee 2-6\\
D-14195 Berlin\\
Germany\\
e-mail: sfroehli@mi.fu-berlin.de\\[3ex]
Frank M\"uller\\
Brandenburgische Technische Universit\"at Cottbus\\
Mathematisches Institut\\
Konrad-Zuse-Stra{\ss}e 1\\
D-03044 Cottbus\\
Germany\\
e-mail: mueller@math.tu-cottbus.de

\end{document}